\documentclass[12pt]{article}
\usepackage{amsmath}
\usepackage{amsfonts}
\usepackage{amssymb}
\usepackage{graphicx}
\usepackage{mathrsfs}
\usepackage{latexsym}
\usepackage{xy} \xyoption{all}
\usepackage{amsthm}
\usepackage{epsfig}
\RequirePackage{color}
\usepackage{multicol}

\newtheorem{theorem}{Theorem}[subsection]

\newtheorem{corollary}[theorem]{Corollary}

\newtheorem{example}[theorem]{Example}

\newtheorem{lemma}[theorem]{Lemma}

\newtheorem{proposition}[theorem]{Proposition}
\newtheorem{remark}[theorem]{Remark}

\begin{document}
\title{The Singular Ideal and the Socle of Incidence Algebras
}
\author{M\"{u}ge KANUN\.{I} $^{(1)}$, \"Ozkay \"OZKAN $^{(2)}$\\$^{(1)\text{ }}$Department of Mathematics, D\"{u}zce University, \\Konuralp 81620 D\"{u}zce, Turkey.\\E-mail: mugekanuni@duzce.edu.tr\\$^{(2)\text{ }}$Department of Mathematics, Gebze Technical University, 
\\41400 Gebze, Kocaeli, Turkey\\E-mail: ozkayoz@yahoo.com\\}
\date{}
\maketitle

\begin{abstract}
Let $R$ be a ring with identity and $I(X,R)$ be the incidence algebra of a
locally finite partially ordered set $X$ over $R.$ In this paper, we compute
the socle and the singular ideal of the incidence ring for some $X$ in terms
of the socle of $R$ and the singular ideal of $R$, respectively.
\end{abstract}

\author{}
\maketitle

\section{Introduction and Preliminaries}

The investigation of a ring or algebra $R$ is usually enriched by
understanding of special types of ideals or subspaces of $R$ such as the
Jacobson radical, the prime radical, the socle, the singular ideal, the
center, etc. Although incidence rings have been interest of study for a few
decades, there does not seem to be any results in the literature on the
socle of incidence algebras.

In this paper, we will be computing the left socle and the left singular
ideal of an incidence algebra ${I(X,R)}$, however similar statements hold
for right equivalents by replacing $Min(X)$ by $Max(X)$ and the upper
finiteness by the lower finiteness notion. The outlay of the paper is as follows: 
We give preliminaries and easily computable observations in the introduction.
In Section 2, we look at the singular ideal of particular incidence algebras. Theorem \ref{SingularMain} states that for any locally finite partially ordered set $X$ and any ring $R$, $ FS(I(X,Sing_l(R))) \subseteq Sing_l({I(X,R)}) \subseteq I(X,Sing_l(R)).$ 

Proposition \ref{MinSingular} computes that the left singular ideal of ${I(X,R)}$ is the incidence subalgebra of $X$ over the left singular ideal of $R$, under the hypothesis that $Min(X)$ is a maximal antichain and $\kappa(x)$ is finite for any $x
\in Min(X)$. Corollary \ref{CorMinSingular} states that if $X$ is finite, then 
$Sing_l({I(X,R)}) = I(X,Sing_l(R))$. 

In Section 3, we look at the socle of incidence algebras.  
Theorem	\ref{MinSocle} states that if $Min(X)$ is a maximal antichain and $\kappa(x)$ is finite for any $x \in 	Min(X) $, then  
	$Soc_l({I(X,R)}) = \bigoplus_{x \in Min(X)} S(x,Soc_l(R)).$ 
	
	Proposition \ref{EmptyMin} computes $Soc_{l}({I(X,R)})=\{0\}$ if 
	$Min(X)=\emptyset $. The main theorem of this section is Theorem \ref{Main} that states $Soc_{l}({I(X,R)})=\bigoplus_{x\in Min(Y)}S(x,Soc_{l}(R))$ under the assumption that $R$ is an Artinian and nonsingular ring, $Min(Y)$ is finite and $\kappa (x)$ is finite for all $x\in X$.  
	We also show that the left and right socles of ${I(X,R)}$ coincide if and only if $R$ satisfies the same property and ${I(X,R)}$ is a product of copies of $R$. At the end of the paper, we provide an example of an incidence algebra with non-equal left and right socles.

Throughout the article, $R$ is a ring with unity and not necessarily
commutative. We assume $X$ is a locally finite partially ordered set. $%
Min(X) $ is the set of all minimal elements of $X$, and $Max(X)$ is the set
of all maximal elements of $X$. $Min(X)$, $Max(X)$ may very well be empty
sets. For any set $U$, we will use $|U|$ to denote the cardinality of $U$.

The \textit{socle} $Soc(M)$ of an $R$-module $M$ is the sum of all its
simple (minimal) submodules. If we take $M=R$ as a left $R$-module, then the
sum of all minimal left ideals of $R$ is \textit{the left socle} of $R$,
denoted by $Soc_l(R)$. Similarly the right socle of $R$, $Soc_r(R)$ is
defined as the sum of all minimal right ideals of $R$. A \textit{left
(right) essential ideal} $E$ of $R$ is a left (right) ideal which intersects
any non-zero left (right) ideal of $R$ non-trivially. It is well-known that
the left (right) socle of a ring is equal to the intersection of all left
(right) essential ideals.

For any subset $E$ of $R$, let $ann_l(E)$, ($ann_r(E)$) denote the left
(right) annihilator of $E$ and for any $a \in R$, $E:a =\{r \in R: ra \in E
\}$. If $E$ is an ideal, then so is $E:a$.

A \textit{left dense} ideal $D$ of a ring $R$ is a left ideal with the
property that for any $a\in R$, $ann_r(D:a)=0.$

It is an easy exercise to show that any left (right) dense ideal is a left
(right) essential ideal. The intersection of any finite collection of left
(right) essential (dense) ideals is always left (right) essential (dense).
However, the intersection of all the left (right) essential ideals of the
ring $R$ is a left (right) essential ideal if and only if there is a minimal
left (right) essential ideal of $R$. Similarly, intersection of all left
(right) dense ideals is a left (right) dense ideal if and only if there is a
minimal left (right) dense ideal of $R$.

For any ring $R$, 
\begin{equation*}
Sing_l(R)=\{x\in R: \, ann_l(x) \text{ is a left essential ideal of } R\}
\end{equation*}
is called \textit{the left singular ideal} of $R$. A ring $R$ is a \textit{%
left nonsingular ring} if $Sing_l(R)=0.$ Similar statements hold for right
singular ideal $Sing_r(R)$ of $R$.

$Soc_r(R)$, $Soc_l(R)$, $Sing_l(R)$ and $Sing_r(R)$ are all two-sided ideals
of $R$.



Let $X$ be a partially ordered set. For any $x,z\in X$ with $x \leq z$, the
interval $[x,z]$ is defined as \newline
\begin{equation*}
[x,z]=\{ y\in X: x \leq y\leq z\}.
\end{equation*}
If every interval of $X$ is a finite set, then $X$ is a locally finite
partially ordered set.

For a ring $R$ and a locally finite partially ordered set $X$, the 
\textit{incidence ring} $I(X,R)$, is the set of functions $f:X\times
X\rightarrow R$ such that $f(x,y)=0$ unless $x\leq y,$ with the following
operations 
\begin{eqnarray*}
(f+g)(x,y) &=&f(x,y)+g(x,y) \\
fg(x,y) &=&\sum\limits_{x\leq z\leq y}f(x,z)g(z,y)
\end{eqnarray*}
for all $f,g\in I(X,R)$ and $x,y\in X.$ Also, $I(X,R)$ becomes an $R$-algebra with the operation: \newline
$(rf)(x,y)=rf(x,y)$ for any $r\in R.$ 
The element $\delta \in I(X,R)$ which is defined as

\begin{center}
$\delta (u,v)=\left\{ 
\begin{array}{cc}
1 & \text{ if }u=v \\ 
0 & \text{ otherwise }%
\end{array}
\right. $
\end{center}
is the identity element of $I(X,R)$. For any $(x,y)\in X\times X$ with $x\leq y$, we define $e_{xy}\in I(X,R)$ as

\begin{center}
$e_{xy}(u,v)=\left\{ 
\begin{array}{cc}
1 & \text{ if }(u,v)=(x,y) \\ 
0 & \text{otherwise }%
\end{array}
\right. .$
\end{center}

Let $supp(g) =\{ (x,y) ~:~ g(x,y)\neq 0\} $ denote the support of $g\in
I(X,R)$. The set 
\begin{equation*}
FS(I(X,R))= \{ g \in I(X,R) ~:~ |supp(g)| < \infty \}
\end{equation*}
is a subring of $I(X,R)$ which is called the \textit{finite support of} $%
I(X,R)$.

For a detailed discussion about incidence rings see
\cite{SO:IncAlg}.


A partially ordered set, $X$ is called \textit{upper finite} if for each $%
x\in X,$ $U(x):=\{y\in X~|~x \leq y\}$ is finite. Similarly, $X$ is called 
\textit{lower finite} if for each $x\in X$, $L(x):=\{y\in X~|~y \leq x\}$ is
finite. For any $x \in X$, define $\kappa(x):= |U(x)|$ and $\lambda(x) : =
|L(x)|$.



A subset $S$ of $X$ is called an \textit{antichain} if for any pair $(x,y)$
of elements in $S$ are incomparable. An antichain $S$ is called a \textit{%
maximal antichain of }$X$ if $S$ is a maximal element in the partially
ordered set of the collection of all antichains of $X$ under inclusion. We
first state a few easily deducible observations that will be used in the
sequel.

\begin{lemma}
\label{propPoset} Assume $X$ is a locally finite partially ordered set. Then

\textrm{(i)} 
$Min(X)$ is a maximal antichain if and only if for each $y\in X$, there
exists a minimal element $x\in Min(X)$ with $x \leq y$.

\textrm{(ii)} 
If $Min(X)$ is a maximal antichain and $\kappa(x)$ is finite for all $x \in
Min(X)$, then $X$ is upper finite.

\textrm{(iii)} 
If $X$ is an antichain, then $Min(X)=Max(X)$.

\textrm{(iv)} 
If $Min(X)=Max(X)$ and $X$ is upper finite, then $X$ is an antichain.
\end{lemma}

\begin{proof}
\textrm{(i)} Clearly, $Min(X)$ is an antichain, as no two minimal elements
are related. Assume $Min(X)$ is a maximal antichain. On the contrary if
there is a $y\in X$ with no $x\in Min(X)$ such that $x \leq y$, then $x$ and 
$y$ are unrelated, hence $Min(X) \cup \{ y\}$ is an antichain. This
contradicts the maximality of $Min(X)$. 
Conversely, assume on the contrary that $Min(X)$ is not a maximal antichain,
hence there exists an antichain $Y$ properly contaning $Min(X)$. Let $y \in
Y \backslash Min(X)$, so $y$ is incomparable with any $x \in Min(X)$ which
contradicts the hypothesis.

\textrm{(ii)} Since $Min(X)$ is an antichain, by \textrm{(i)}, for each $%
y\in X $, there exists a minimal element $x\in Min(X)$ with $x \leq y$. As $%
\kappa(y) \leq \kappa(x)$, the result follows.

\textrm{(iii)} If $X$ is an antichain, then every element in $X$, is both a
maximal and a minimal element. So $Min(X)=Max(X)=X$. 

\textrm{(iv)} Assume not, then there exists two distinct elements $x,y$ such
that $x \leq y$. Since $X$ is upper finite, $U(x) = \{z \in X \ : \ x \leq z
\}$ is a finite set, say $\{y_0:=x,y_1:=y,..., y_n \}$. Notice that if $y_i
\in U(x)$, for $i \neq 0$, then $U(y_i) \subsetneq U(x)$. Then there exists $%
z \in U(x)$ ($z \neq x$) which is also in $Max(X)$. Since $Max(X)=Min(X)$, $%
z\in Min(X)$ which is a contradiction to $x \leq z$.
\end{proof}

\medskip


Note that the converses of Lemma \ref{propPoset} \textrm{(ii)}, (iii) are not true. For instance, the partially ordered set $X$ in Example~\ref{UpfiniteNotMinXmax}
is upper finite, but $Min(X)$ is not a maximal antichain.


\begin{example}
\label{UpfiniteNotMinXmax} Let $X =\{x_i,y_i~:~ i \in \mathbb{N}\}$ be a
partially ordered set with the relations $x_i \leq y_1 $, $%
y_{i+1} \leq y_i $ for all $i \in \mathbb{N}$. The Hasse diagram of $X$ is 
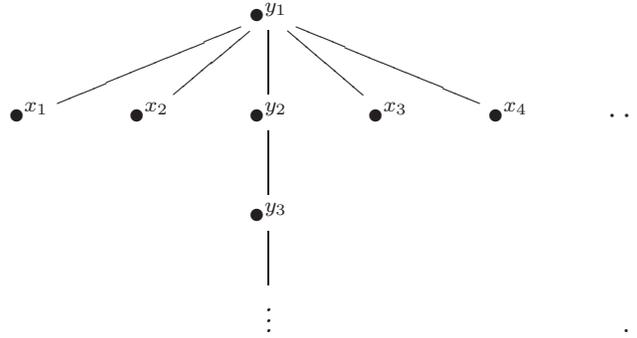
\begin{equation*}
\xymatrix{&&{\bullet}^{y_1}&&& \\
{\bullet}^{x_1}\ar@{-}[urr]&{\bullet}^{x_2}\ar@{-}[ur]&{\bullet}^{y_2}%
\ar@{-}[u]&{\bullet}^{x_3}\ar@{-}[ul]& {\bullet}^{x_4}\ar@{-}[ull]&\cdots \\
&&{\bullet}^{y_3}\ar@{-}[u]\ar@{-}[d]&&& \\ &&\vdots &&& .}
\end{equation*}

Hence, $Min(X) = \{x_i~:~ i \in \mathbb{N}\}$ and $Max(X) = \{y_1 \}$. Now, $%
Min(X)$ is a non-empty antichain, which is not maximal, since $Min(X) \cup
\{y_i\}$ is a maximal antichain for any $i \geq 2$. Moreover, $\kappa(x)$ is
finite for all $x \in X$, that is $X$ is upper finite.

\end{example}

\medskip

\begin{remark}
\label{remarkAli} \textrm{When $X$ is finite, for any $y$ in $X$, there is $%
x $ in $Min(X)$ with $x \leq y$. By Lemma~\ref{propPoset}(i), $Min(X)$ is a
maximal antichain. }
\end{remark}



\medskip

Let $X$ be a locally finite partially ordered set with a subset $Y$ defined as:
\begin{equation*}
Y =\{y \in X : \mbox{ there is a }x \in Min(X) \mbox{ with } x \leq y\}.
\end{equation*}
Then,  $X = Y \sqcup (X \backslash Y)$ and let $Z = X \backslash Y$. Now, we mention some properties of $%
X,Y $ and $Z$ in Lemma \ref{Veli}.

\begin{lemma}
\label{Veli} 1. Let $X$ be a locally finite partially ordered set, $Y$ and $%
Z $ be subposets defined as above. Then $X = Y \sqcup Z$, (ie. $X$ is a
disjoint union of $Y$ and $Z$).

2. $Min(Y)=Min(X)$ and $Min(Z) = \emptyset$.

3. $Min(Y)$ is a maximal antichain of $Y$.

4. $Min(Y)$ is a maximal antichain of $X$ if and only if $Z = \emptyset$.
\end{lemma}

\begin{proof}
1. Assume $Min(X)\neq \emptyset $. There exists a $x\in Min(X)$ with $x\leq
x $, by definition of $Y$, we get $x \in Y$, then $Y\neq \emptyset$. Assume
there exists $a\in X$ but $a\notin Min(X)$. We have either $a\geq x$ or $a$
and $x$ are unrelated for some $x\in Min(X)$. For the case $a\geq x$, $a$
must be in $Y.$ Otherwise, if $a$ and $x$ are unrelated, $a\in X\ \backslash
Y$\textbf{\ }. For all such $a\in X$, either $a\in Y$ or $a\in X\ \backslash
Y$.

On the other hand, if $Min(X)=\mathbf{\emptyset }$, then $Y=\mathbf{%
\emptyset }$ and $X=Z$.

2. $\mathbf{Min(X)\subseteq Min(Y)}:$ Let $\beta \in Min(X)\subset Y$ and
assume that $\beta \notin Min(Y)$. By definition of the set $Y$, there
exists a $\alpha \in Min(Y)$ with $\alpha \leq \beta $. Since $\alpha \in Y$%
, there exists $\gamma \in Min(X)$ with $\gamma \leq \alpha \leq \beta $.
This contradicts to minimality of $\beta $. Hence $\beta \in Min(Y)$.

$\mathbf{Min(Y)\subseteq Min(X)}:$ Let \ $\alpha \in Min(Y)$. Since $%
Min(Y)\subseteq Y$ and using the result $Min(X)\subseteq Min(Y)$, there
exists $\beta \in Min(X)$ with $\beta \leq \alpha $. Since $\beta \in Min(X)$
then $\beta \in Y$. Hence $\beta \leq \alpha $ and $\alpha \leq \beta $.
That is $\alpha =\beta $, then $\alpha \in Min(X)$.

$\mathbf{Min(Z)=\emptyset }$ $:$ Assume there exists an element $z\in Min(Z)$%
. Since $Min(Z)\subseteq Z$, $z\in Z=X\ \backslash Y$, that is $z\notin Y$.
Note that $z\notin Min(X)$, so there exists $x_{1}\in X$ with $x_{1}\neq z$
such that $x_{1}\leq z.$

We claim that $x_{1}\in Z$ i.e $x_{1}\notin Y$. Otherwise, if $x_{1}\in Y$,
there exists $x_{2}\in Min(X)$ such that $x_{2}\leq x_{1}\leq z$ which
implies $z\in Y$, a contradiction.

3. Consider the definition of $Y$. For each $y\in Y$ there exists a minimal
element $x\in Min(X)$ with $x\leq y$. By (2), $Min(Y)=Min(X)$. So $x\in
Min(Y)$. By Lemma \ref{propPoset}(i), $Min(Y)$ is a maximal antichain of $Y.$

4. Assume $Min(Y)$ is not a maximal antichain. There
exists an element $z\in X$ incomparable with any element of $Y$. That is $%
z\in X\ \backslash Y$, $Z\neq \mathbf{\emptyset }$.

Assume $Z=\mathbf{\emptyset }$. Then by (1), $X=Y$.
Using definition of $Y$, there exists $x\in Min(X)$ such that $x\leq y$.
Thus $Min(X)\neq \mathbf{\emptyset }$. Since $X=Y$, there exists $y\in Y$
such that $y\leq x$. By Lemma\ref{propPoset}(i) $Min(Y)$ is a maximal
antichain of $Y=X$.
\end{proof}

\section{Singular Ideal of Incidence Algebras}


In this section, we will investigate the left singular ideal of an incidence algebra.  
Moreover, similar arguments that appear here can be used to achieve analogue results for the right singular ideal.

We start by quoting some results that appeared in literature that is needed. 
The following lemma is from \cite{Trio:CharIncAlg}, we give the proof for
completeness of the argument in the sequel. 

\begin{lemma}
\label{lemma 1.3} (\cite[Lemma 11]{Trio:CharIncAlg}) Assume $r$ is a
non-zero element of $R$, and $x,y\in X$ with $x\leq y. $ Then $%
ann_l(re_{xy}) $ is a left essential ideal of $I(X,R)$ if and only if $%
ann_l(r)$ is a left essential ideal of $R.$
\end{lemma}

\begin{proof}
First, note that $ann_l(re_{xy})=\{g\in I(X,R): \, g(u,x)\in ann_l(r) \text{
for all } u \in X\}$.

Assume $ann_l(r)$ is a left essential ideal of $R.$ Take any non-zero $f\in
I(X,R)$.

Case 1: If $f(u,x)=0$ for all $u\in X,$ then $f\in ann_l(re_{xy}).$ Thus $%
ann_l(re_{xy})\cap I(X,R)f\neq \{0\}.$

Case 2: Otherwise, there exists $z\in X$ such that $f(z,x)\neq 0.$ Let $S$
be the set of all $z \in X$ such that $f(z,x) \neq 0$. Then $ann_l(r)\cap
Rf(z,x)\neq \{0\}.$ So there exists $r_{z}\in R$ such that $r_{z}f(z,x)\in
ann_l(r),$ $r_{z}f(z,x)\neq 0.$ Now, construct $g\in I(X,R)$ where 
\begin{equation*}
g(u,v)=\left\{ 
\begin{array}{ccc}
r_{z} & if & u \leq z~~and~~v = z \\ 
0 &  & otherwise%
\end{array}
\right. .
\end{equation*}
For any $u\in X,$ $S \cap [u,x]$ is a finite set, as $X$ is locally finite,
so 
\begin{equation*}
gf(u,x)=\sum_{t}g(u,t)f(t,x)= \sum_{z \in S \cap [u,x]}g(u,z)f(z,x)= \sum_{z
\in S \cap [u,x]} r_{z}f(z,x) \in ann_l(r).
\end{equation*}
So $gf \in ann_l(re_{xy})$. Now, we show $gf \neq 0$. There is a $z \in S$,
and an interval $[z,x]$ such that $S \cap [z,x] = \{ z \}$. Then $%
gf(z,x)=g(z,z)f(z,x) = r_z f(z,x) \neq 0$.

Then $gf\in ann_l(re_{xy})\cap I(X,R)f.$ Hence $ann_l(re_{xy})$ is a left
essential ideal of $I(X,R).$

Conversely, assume $ann_l(re_{xy})$ is a left essential ideal of $I(X,R).$
Take any non-zero $a \in R.$ Then $I(X,R)(a e_{xx})\cap ann_l(re_{xy})\neq
0, $ so there exists $f\in I(X,R)$ such that $fa e_{xx}\neq 0$ and also $fa
e_{xx}\in ann_l(re_{xy}).$ Hence there exists $u\in X$ such that $%
f(u,x)a\neq 0$ and also $f(u,x)a\in ann_l(r).$ Therefore, $Ra \cap
ann_l(r)\neq 0$ and $ann_l(r)$ is left essential.
\end{proof}

\bigskip

Now, we are ready to prove that for any $X$ and $R$, the finite
support of $I(X,Sing_l(R))$ is contained in the singular ideal of ${I(X,R)}$.

\begin{proposition}
\label{finitesupport} For any partially ordered set $X$ and ring $R$, 
$$ FS(I(X,Sing_l(R))) \subseteq Sing_l({I(X,R)}).$$
\end{proposition}

\begin{proof}
Let $f$ be in the finite support of $I(X,Sing_l(R)))$, then $f = \sum
r_{xy}e_{xy}$ where $r_{xy} \in Sing_l(R)$ and the sum is over finitely many
entries $(x,y)$ in $X \times X$, $x \leq y$. For any $(x,y)$ with $x \leq y$%
, $ann_l(r_{xy})$ is a left essential ideal of $R$. By Lemma \ref{lemma 1.3} 
$ann_l(r_{xy}e_{xy})$ is left essential ideal of ${I(X,R)}$, so $%
r_{xy}e_{xy} $ is in the left singular ideal of ${I(X,R)}$. Hence, $f \in
Sing_l({I(X,R)})$.
\end{proof}

\medskip

For any $X$ and $R$, $Sing_l({I(X,R)}) \subset I(X,Sing_l(R)) $, which is
already proved in \cite[Thm 12, Prop 13]{Trio:CharIncAlg}. We include the
proof here for the completeness of the argument.

\begin{proposition}
\label{Singularideal} Assume $X$ is a locally finite partially ordered set, $%
R$ is any ring. Then 
\begin{equation*}
Sing_l({I(X,R)}) \subset I(X,Sing_l(R)) .
\end{equation*}
\end{proposition}

\begin{proof}
Take $f\in Sing_l({I(X,R)})$ and $f\neq 0,$ then $ann_l(f)$ is left
essential. 
\begin{equation*}
ann_l(f)=\bigcap_{x\in X}ann_l(fe_{xx})\subseteq ann_l(fe_{xx}) \text{ for
all } x\in X.
\end{equation*}
Hence, $ann_l(fe_{xx})$ is a left essential ideal for all $x\in X.$

Since $f\neq 0$, there exists $f(x,y)\neq 0$ for some $(x,y)\in X\times X.$
Let $r \in R,~r \neq 0$ be arbitrary. Construct $r e_{xx}\in I(X,R)$. Then 
\begin{equation*}
ann_l(fe_{yy})\cap I(X,R)(r e_{xx})\neq \{0\}.
\end{equation*}
So there exists $h\in I(X,R)$ such that $h(r e_{xx})\neq 0$ and $hre_{xx}\in
ann_l(fe_{yy}).$ That is, there exists $t\leq x$ such that 
\begin{equation*}
(hre_{xx})(t,x)= h(t,x)r\neq 0
\end{equation*}
and 
\begin{equation*}
\sum_{v\leq z\leq y}(hre_{xx})(v,z)f(z,y)=0 \text{ for all } v\leq x.
\end{equation*}
For $v=t,$ 
\begin{equation*}
0=\sum_{t\leq z\leq y}(hre_{xx})(t,z)f(z,y)=h(t,x)rf(x,y).
\end{equation*}
This implies $h(t,x)r\in ann_l(f(x,y)).$ So $h(t,x)r\in Rr\cap
ann_l(f(x,y))\neq \{0\}.$ Then $ann_l(f(x,y))$ is left essential in $R$. So $%
f(x,y)\in Sing_l(R).$ Hence $f\in I(X,Sing_l(R)).$
\end{proof}

\bigskip
Propositions \ref{finitesupport}  and \ref{Singularideal} give the main result on the singular ideal of an incidence algebra.  
\begin{theorem}
\label{SingularMain} For any locally finite partially ordered set $X$ and any ring $R$, 
$$ FS(I(X,Sing_l(R))) \subseteq Sing_l({I(X,R)}) \subseteq I(X,Sing_l(R)).$$
\end{theorem}

\bigskip

We turn our attention to sharpening the inclusion in Theorem \ref{SingularMain} under 
some restrictions of $X$ or $R$ in the rest of this section. We analyze the conditions needed for $Sing_l({I(X,R)}) = I(X,Sing_l(R))$.

In \cite[Theorem 12]{Trio:CharIncAlg}, it is shown that ${I(X,R)}$ is (left)
nonsingular if and only if $R$ is left non-singular. Although this result is a very simple observation, as we have quoted the statement throughout the text, we will state it as a proposition.
\begin{proposition}(\cite[Theorem 12]{Trio:CharIncAlg})\label{Nonsingular} 
$Sing_l({I(X,R)})= 
\{0\}$ if and only if $Sing_l(R)= \{0\}$ . Hence, in this case, 
\begin{equation*}
Sing_l({I(X,R)}) = I(X,Sing_l(R)) = \{0\} .
\end{equation*}
\end{proposition}

\bigskip

Now, we consider the left singular ideal of an incidence algebra over a partially
ordered set $X$ with $Min(X)$ a maximal antichain and $\kappa(x)$ finite. We will need the dual version of the following lemma from \cite{Singh:MaxQuotientIncAlg} that we restate here. 


\begin{lemma}
\label{Lemma1-2Singh}(\cite[Lemma 1, Lemma 2]{Singh:MaxQuotientIncAlg})
Assume $A$ is left ideal of ${I(X,R)}$. For any $x,y \in X$ with $x \leq y$,

\textrm{(i)} $A(x, y) =\{f(x, y) \in R : f \in A\}$ is a left ideal of $R$.

\textrm{(ii)} For $x \leq y \leq z$ in $X$, $A(x, y) \subset A(x, z)$.

\textrm{(iii)} If $A$ is an essential left ideal and $x \in Min(X)$, then\newline
$C(x,y)= \{f(x,y) \in R : f \in Re_{xy} \cap A\}$ is an essential left ideal
of $R$ and $C(x,y) \subset A(x, y)$.

\textrm{(iv)} If $A$ is an essential left ideal then there exists a family $\mathcal{%
E}$ of essential left ideals of $R$ such that for any $x \in Min(X)$, and $y
\in X$ with $x \leq y$, there exists $E_{xy} \in \mathcal{E}$ with $%
E_{xy}e_{xy} \subset A$.

\textrm{(v)} Assume $Min(X)$ is a maximal antichain and $\kappa(x)$ is finite for
all $x \in X$. If there exists a family $\mathcal{E}$ of essential left
ideals of $R$ such that for any $x \in Min(X)$, and $y \in X$ with $x \leq y$%
, there exists $E_{xy} \in \mathcal{E}$ with $E_{xy}e_{xy} \subset A$. Then $%
A$ is an essential left ideal.
\end{lemma}

\medskip

%

\medskip

Here is the result we seeked for: 

\begin{proposition}
\label{MinSingular} Assume $X$ is a locally finite partially ordered set
where $Min(X)$ is a maximal antichain and $\kappa(x)$ is finite for any $x
\in Min(X)$. Then, the left singular ideal of ${I(X,R)}$ is the incidence
subalgebra of $X$ over the left singular ideal of $R$. That is, 
\begin{equation*}
Sing_l({I(X,R)}) = I(X,Sing_l(R)) .
\end{equation*}
\end{proposition}

\begin{proof}
By Proposition \ref{Singularideal}, one containment is satisfied.

For the converse, take a non-zero element $f \in I(X,Sing_l(R))$. We want to
show that $ann_l(f)$ is a left essential ideal in ${I(X,R)}$. By using Lemma~%
\ref{Lemma1-2Singh}, it is enough to show that there exists a collection $%
\mathcal{E}=\{A_{xy}: x \in Min(X), \, x \leq y\}$ of left essential ideals
in $R$, such that $A_{xy}e_{xy} \subset ann_l(f)$ for all $x \in Min(X)$, $x
\leq y$. Pick any $x,y$ with $x \in Min(X)$ and $x \leq y$. Define 
\begin{equation*}
A_{xy} = \bigcap_{v \in U(y)} ann_l(f(y,v)).
\end{equation*}
Since $f(y,v) \in Sing_l(R)$, $ann_l(f(y,v))$ is left essential ideal of $R$
for all $v \in U(y)$. As, $|U(y)| = \kappa(y)$ is a finite set by Lemma \ref%
{propPoset} (ii), $A_{xy}$ is the finite intersection of left essential
ideals. Hence, $A_{xy}$ is left essential. 
Moreover, take any $r \in A_{xy} = \bigcap ann_l(f(y,v))$, to complete the
proof we show that $re_{xy} \in A_{xy}e_{xy}$ is in the left annihilator of $%
f$. Take any $(u,v)$,

\begin{equation*}
(re_{xy}f) (u,v) = 
\begin{cases}
rf(y,v) & \text{ if } u=x, \, y \leq v \\ 
0 & \text{ otherwise }%
\end{cases}
= 
\begin{cases}
0 & \text{ if } u=x, \, y \leq v \\ 
0 & \text{ otherwise }%
\end{cases}%
\quad = 0.
\end{equation*}
Hence $A_{xy}e_{xy}f = 0$ and $A_{xy}e_{xy} \subset ann_l(f)$.
\end{proof}

\medskip 

When $X$ is finite, $Min(X)$ is a maximal antichain by Remark \ref{remarkAli}. Clearly, any finite $X$ is upper finite. Hence, we get the
following corollary which is already stated in \cite[Proposition 13]{Trio:CharIncAlg}.

\begin{corollary}
\label{CorMinSingular} Assume $X$ is finite, then 
\begin{equation*}
Sing_l({I(X,R)}) = I(X,Sing_l(R)) .
\end{equation*}
\end{corollary}

Another proof of Corollary \ref{CorMinSingular} is that the finite support
of $I(X,Sing_l(R))$ is $I(X,Sing_l(R))$ itself, when $X$ is finite. By
Proposition \ref{finitesupport} result follows.

\medskip


%



\section{Socle of Incidence Algebras}


In this section, we will investigate the socle of an incidence algebra. In order to calculate the left socle, we will look at the
essential ideals of an incidence algebra, which are already studied in \cite{S:Essential}. Also, any dense ideal is essential, so the description of various dense ideals 
of an incidence algebra defined in the paper \cite{K:Dense} will be useful.

The following constructions are from \cite{K:Dense}. 

For any $x \in Min(X)$, fix a left ideal $I_x$ of $R$ and define a subset $%
S(x,I_x)$ of ${I(X,R)}$ as 
\begin{equation*}
S(x,I_x) = \{f \in I(X,I_x) : \, f(u,v) = 0 \mbox{ if } u \neq x \}.
\end{equation*}
Then the set 
\begin{equation*}
C({I(X,R)}) := \bigoplus_{x \in Min(X)}S(x,I_x) = \bigoplus_{x \in Min(X)}
\{f \in I(X,I_x) : \, f(u,v) = 0 \mbox{ if } u \neq x \}
\end{equation*}
is a two-sided ideal of ${I(X,R)}$.

%
%



\begin{lemma}
\label{C(I)} Assume $X$ is a locally finite partially ordered set where $%
Min(X)$ is a maximal antichain and $\{I_x : x \in Min(X)\}$ is a collection
of left essential ideals of $R$. 

\textrm{(i)} $C({I(X,R)})$ is a left essential ideal of ${I(X,R)}$. \newline
\textrm{(ii)} Let $C$ be the intersection of all left essential ideals of
the form $C({I(X,R)})$. Then 
\begin{equation*}
C = \bigoplus_{x \in Min(X)} S(x,Soc_l(R)).
\end{equation*}
\textrm{(iii)} $Soc_l({I(X,R)}) \subset C$. 
\end{lemma}
\begin{proof}
\textrm{(i)} Assume $Min(X)$ is a maximal antichain. Then left ideals of the
form $C({I(X,R)})$ are left essential. \newline
\textrm{(ii)} If $f \in C$, then $f = f_{x_1} + f_{x_2} + \cdots f_{x_n}$
where $f_{x_i} \in S(x_i, D_i)$ for any left essential ideal $D_i$ where $%
x_i \in Min(X)$, $i = 1, \cdots, n$. So, $f_{x_i} \in S(x_i, Soc_l(R))$ and $%
f \in \bigoplus_{x \in Min(X)} S(x,Soc_l(R))$. Similarly, $f \in
\bigoplus_{x \in Min(X)} S(x,Soc_l(R))$ implies $f$ is in any left essential
ideal of the form $C({I(X,R)})$, so $f \in C$. \newline
\textrm{(iii)} Since, $Soc_l({I(X,R)})$ is the intersection of all left
essential ideals, result follows. 
\end{proof}

When we consider the assumption that $Min(X)$ is a maximal antichain and 
$\kappa(x)$ is finite for any $x \in Min(X)$, we achieve the equality, 
i.e. $$Soc_l({I(X,R)}) = \bigoplus_{x \in Min(X)} S(x,Soc_l(R)).$$

\begin{theorem}
	\label{MinSocle} Assume $X$ is a locally finite partially ordered set where 
	$Min(X)$ is a maximal antichain and $\kappa(x)$ is finite for any $x \in
	Min(X) $. Then, 
	\begin{equation*}
	Soc_l({I(X,R)}) = \bigoplus_{x \in Min(X)} S(x,Soc_l(R))
	\end{equation*}
\end{theorem}

\begin{proof}
Assume $Min(X)$ is a maximal antichain, then Lemma~\ref{C(I)}(iii) gives one inclusion. Also, by hypothesis $\kappa(x)$ is finite for any $x\in Min(X)$. Now, let $f\in \bigoplus_{x \in Min(X)} S(x,Soc_l(R))$, and $K$ be any left essential ideal of $I(X,R)$. 

Now, $\displaystyle f=\sum_{x \in  \{x_1,x_2, ..., x_n \} \subset Min(X)}f_{x}$ where $f_{x}\in S(x_{i},Soc_{l}(R))$ for
some $x_{i}\in Min(X)$, $i=1,\cdots ,n$. Let $x\in \{x_{1},\cdots,x_{n}\} $ then 
\begin{equation*}
f_{x}=\sum_{x\leq y} \alpha_{y}e_{xy}\text{ where }%
\alpha_{y}\in Soc_{l}(R).
\end{equation*}%
Since $\kappa(x)< \infty$, this sum is finite. Then 
\begin{equation*}
C(x,y)=\{r\in R: re_{xy}\in Re_{xy}\cap K \} 
\end{equation*} is a left essential ideal of $R$ by Lemma~\ref{Lemma1-2Singh}(iii). So $ Soc_{l}(R)\subseteq C(x,y)$. 
Since $\alpha_{y}\in Soc_{l}(R)\subseteq C(x,y)$, 
\begin{equation*}
\alpha_{y}e_{xy}\in C(x,y)e_{xy}=Re_{xy}\cap K\subseteq K
\end{equation*}
and $f_{x}\in K$, hence $f\in K$. For any left essential ideal $K$, 
\begin{equation*}
\bigoplus_{x \in Min(X)} S(x,Soc_l(R))\subseteq  K, 
\end{equation*}
and
\begin{equation*}
 \bigoplus_{x \in Min(X)} S(x,Soc_l(R))\subseteq \bigcap_{ \mbox{all essential } K} K = Soc_l({I(X,R)}).  
\end{equation*}
This completes the proof. 
\end{proof}

We provide an example for the previous theorem.
\begin{example} 
	The partially ordered set $\bigcup C_n$ is defined to be the set $X =\{x_{11},x_{21},x_{22},x_{31},x_{32},x_{33},x_{41},\cdots \}$
with the relation that $x_{ij} \leq x_{kl}$ whenever $i=k$ and $j \leq l$. The $\bigcup C_n$ is an example of unbounded partially ordered set with no infinite chain.
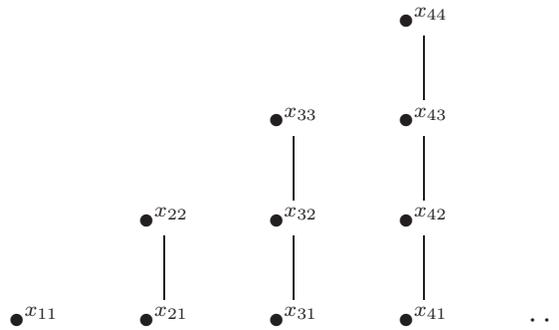
\begin{equation*}
	     \xymatrix{&&&&&{\bullet}^{x_{44}}& \\
		&&&&{\bullet}^{x_{33}} &{\bullet}^{x_{43}}\ar@{-}[u]& \\
		&&&{\bullet}^{x_{22}} & {\bullet}^{x_{32}}\ar@{-}[u]&{\bullet}^{x_{42}}\ar@{-}[u]& \\
		&&{\bullet}^{x_{11}} & {\bullet}^{x_{21}}\ar@{-}[u] &{\bullet}^{x_{31}}\ar@{-}[u] &{\bullet}^{x_{41}}\ar@{-}[u] &\cdots }	
\end{equation*}
By Theorem \ref{MinSocle}, 
\begin{equation*}
Soc_l({I(X,R)}) = \bigoplus_{x \in Min(X)} S(x,Soc_l(R)).
\end{equation*}
\end{example}
In \cite[Lemma 4]{K:Dense}, for each $n\in \mathbb{N}$, a
collection of left dense ideals $Z_{n}$ is constructed as follows: 
\begin{equation*}
Z_{n}:=\{f\in I(X,R):\;f(x,y)=0\mbox{ if }|[x,y]|\leq n,\mbox{ and }x\not\in
Min(X)\}.
\end{equation*}%
The following lemma is proved in \cite[Lemma 4]{K:Dense}.

\begin{lemma}\label{Zn}
	$Z_{n}$ is a two-sided ideal of ${I(X,R)}$.
	
		\textrm{(i)} $Z_{n}$ is a left essential ideal of ${I(X,R)}$.
		
		\textrm{(ii)} Let $Z$ be the intersection of all $Z_{n}$. Then $Soc_{l}({I(X,R)})\subset Z$ where 
		\begin{equation*}
		Z=\bigcap_{n\in \mathbb{N}}Z_{n}=\{f\in {I(X,R)}:f(x,y)=0\mbox{ for
			all }y\in X,x\not\in Min(X)\}
		\end{equation*}
		\begin{equation*}
		= \prod_{x\in Min(X)}S(x,R).
		\end{equation*}
		\textrm{(iii)} If $Min(X)$ is a finite set, then $\displaystyle Z=\bigoplus_{x\in Min(X)}S(x,R).$
		
		\textrm{(iv)}If $Min(X)=\emptyset $, then $Z=\{0\}.$
\end{lemma}
\begin{proof}
	It is shown in \cite[Lemma 4]{K:Dense} that $Z_{n}$ is a left dense ideal of 
	${I(X,R)}$. Any dense ideal is essential. The rest is trivial. 
\end{proof}

\begin{proposition}
	\label{EmptyMin} Let $R$ be any ring, $X$ be a locally finite partially
	ordered set where $Min(X)=\emptyset $. Then, 
	\begin{equation*}
	Soc_{l}({I(X,R)})=\{0\}.
	\end{equation*}
\end{proposition}
\begin{proof}
	By Lemma~\ref{Zn} (ii) and (iv), the result follows.
\end{proof}

In the following example, $X$ is upper finite, but there are no minimal
elements.
\begin{example}
	Let $X=\mathbb{Z}^-$, $R$ be any ring. Then $Min(X) = \emptyset$. By
	Proposition~\ref{EmptyMin}, $Soc_l({I(X,R)}) = \{ 0\}$.
\end{example}
Our next theorem will restrict $R$ to calculate the socle of the incidence algebra. Hence, we define a collection of sets $D_n$ as follows: 
\begin{equation*}
D_n := \{ f \in I(X,Soc_l(R)): \; f(x,y)=0 \mbox{ if } |[x,y]| \leq n, %
\mbox{ and } x \not \in Min(X)\}.
\end{equation*}

\begin{lemma}\label{Dn}
		\textrm{(i)} For each $n\in \mathbb{N}$, $D_{n}$ is a two-sided ideal of ${I(X,R)}$.
		
		\textrm{(ii)} If $R$ is Artinian and nonsingular, then for each $n\in \mathbb{N}$, $D_{n}$ is a left essential ideal of ${I(X,R)}$.
		
		\textrm{(iii)} If $R$ is Artinian and nonsingular, for $D=\bigcap_{n\in \mathbb{N}}D_{n}$, then we have 
		
		$Soc_{l}({I(X,R)})\subset D$. More precisely, 
		\begin{equation*}
		D=\{f\in {I(X,Soc_{l}(R))}:f(x,y)=0
		\mbox{ for
			all }y\in X,x\not\in Min(X)\}
		\end{equation*}
		\begin{equation*}		
			=\prod_{x\in Min(X)}S(x,Soc_{l}(R)).
		\end{equation*}
		
		\textrm{(iv)} If $R$ is Artinian and nonsingular and $Min(X)$ is a finite set, then 
		$$D=\bigoplus_{x\in
			Min(X)}S(x,Soc_{l}(R)).$$
	\end{lemma}

\begin{proof}
	(i) For any $g\in I(X,R)$ and $f\in D_{n}$, $(gf)(x,y)=\underset{x\leq z\leq
		y}{\sum }g(x,z)f(z,y).$ For $x\notin Min(X)$ and $\left\vert \left[ x,y%
	\right] \right\vert \leq n$, we also have $f(x,y)=0.$ For all $z\leq y,$
	since $z\notin Min(X)$ and $\left\vert \left[ z,y\right] \right\vert \leq n$
	we get $f(z,y)=0.$ Then $(gf)(x,y)=0$ and $gf\in D_{n}$.
	
	(ii) We first show that $D_{n}$ is a left dense ideal. Let $(l,z)\in
	X\times X$ with $l\leq z.$
	
	Case 1: Assume that there exists $x\in Min(X)$ such that $x\leq l\leq z.$ For all $\alpha \in Soc_{l}(R), \alpha e_{xl} \in D_{n}$ and for any 
	$g\in Ann_{r}(D_{n})$, 
	\begin{equation*}
	(\alpha e_{xl}g)(x,z)=0 
	\end{equation*}
	then
	\begin{equation*}
	0=(\alpha e_{xl}g)(x,z)=\underset{x\leq y\leq
		z}{\sum }\alpha e_{xl}(x,y)g(y,z)=\alpha g(l,z)
	\end{equation*}
	Hence
	\begin{equation*}
	g(l,z)\in Ann_{r}(Soc_{l}(R))
	\end{equation*}
	Since $R$ is nonsingular, $Soc_{l}(R))=R$ and $Ann_{r}(Soc_{l}(R))=0$. Hence,
	\begin{equation*}
	g(l,z)=0
	\end{equation*}
	Case  2: Now assume that there in no $x\in Min(X)$ with $x\leq l$. For the case $\left\vert \left[ l,z\right]
	\right\vert >n$, take $x_{1}\in X$ such that $\left\vert \left[ x_{1},l %
	\right] \right\vert >n$ with $g\in Ann_{r}(D_{n})$ and for all $\alpha
	\in Soc_{l}(R)$, 
	\begin{equation*}
	(\alpha e_{x_{1}l}g)(x,z)=0
	\end{equation*}
	We again have
	\begin{equation*}
	0=(\alpha e_{xl}g)(x,z)=\underset{x\leq y\leq
		z}{\sum }\alpha e_{xl}(x,y)g(y,z)=\alpha g(l,z)
	\end{equation*}
	then by the same manner
	\begin{equation*}
	g(l,z)=0
	\end{equation*}
	Hence $Ann_{r}(D_{n})=0$ which means that $D_{n}$
	is a dense ideal of $I(X,R)\dot{)}$. Since a dense ideal is essential, $%
	D_{n} $ is essential. Further, $R$ is non-singular, hence all essential
	ideals are dense.
	
	(iii) \& (iv) are obvious.
\end{proof}


\begin{theorem} \label{Main} 
	Let $X$ be a locally finite partially ordered set with $Y=\{y\in X: \text {there exists }  x\in Min(X) \text{ with }  x\leq y \}$  so that $X=Y\sqcup (X\setminus Y)$. Assume that $R$ is an Artinian and nonsingular ring, $Min(Y)$ is finite and $\kappa (x)$ is finite 
	for all $x\in X$, then \textbf{\ 
		\begin{equation*}
		Soc_{l}({I(X,R)})=\bigoplus_{x\in Min(Y)}S(x,Soc_{l}(R)).
		\end{equation*}
	}	
\end{theorem}	
\begin{proof}
By Lemma \ref{Veli}(ii), $Min(X)=Min(Y)$ and is a finite set, so by Lemma~\ref{Dn}(iii) and (iv), $Soc_{l}({I(X,R)}) \subseteq \bigoplus_{x\in Min(Y)}S(x,Soc_{l}(R))$. Also, by hypothesis $\kappa(x)$ is finite for any $x\in X$. Now, let $f\in \bigoplus_{x \in Min(Y)} S(x,Soc_l(R))$, and $K$ be any left essential ideal of $I(X,R)$. 

Now, $\displaystyle f=\sum_{x \in  \{x_1,x_2, ..., x_n \} \subset Min(Y)}f_{x}$ where $f_{x}\in S(x_{i},Soc_{l}(R))$ for
some $x_{i}\in Min(Y)$, $i=1,\cdots ,n$. Let $x\in \{x_{1},\cdots,x_{n}\} $ then 
\begin{equation*}
f_{x}=\sum_{x\leq y} \alpha_{y}e_{xy}\text{ where }%
\alpha_{y}\in Soc_{l}(R).
\end{equation*}%
Since $\kappa(x)< \infty$, this sum is finite. Then 
\begin{equation*}
C(x,y)=\{r\in R: re_{xy}\in Re_{xy}\cap K \} 
\end{equation*} is a left essential ideal of $R$ by Lemma~\ref{Lemma1-2Singh}(iii). So $ Soc_{l}(R)\subseteq C(x,y)$. 
Since $\alpha_{y}\in Soc_{l}(R)\subseteq C(x,y)$, 
\begin{equation*}
\alpha_{y}e_{xy}\in C(x,y)e_{xy}=Re_{xy}\cap K\subseteq K
\end{equation*}
and $f_{x}\in K$, hence $f\in K$. For any left essential ideal $K$, 
\begin{equation*}
\bigoplus_{x \in Min(Y)} S(x,Soc_l(R))\subseteq  K, 
\end{equation*}
and
\begin{equation*}
 \bigoplus_{x \in Min(Y)} S(x,Soc_l(R))\subseteq \bigcap_{ \mbox{all essential } K} K = Soc_l({I(X,R)}).  
\end{equation*}
This completes the proof. 
\end{proof}


\begin{example}
	Let $R$ be an Artinian nonsingular ring and $X =Y \sqcup Z$ where $Y=\{z_1,y_1,y_2,y_3,y_4\}$ and $Z= \{z_i~:~ i \in \mathbb{N}, \, \, i \geq 2 \}$ be a
	partially ordered set with the relations $y_i \leq z_1 $, for $1\leq i \leq  4$ and  $z_{i+1} \leq z_i $ for all $i \in \mathbb{N}$. The Hasse diagram of $X$ is 
	
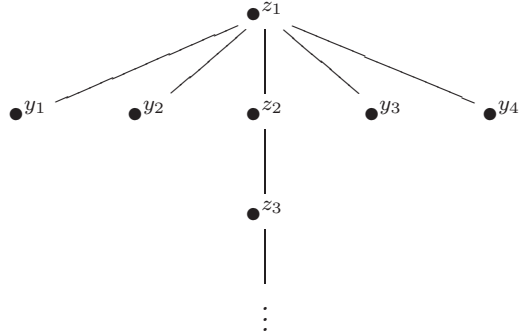
\begin{equation*}	
	\xymatrix{&&{\bullet}^{z_1}&&& \\
		{\bullet}^{y_1}\ar@{-}[urr]&{\bullet}^{y_2}\ar@{-}[ur]&{\bullet}^{z_2}%
		\ar@{-}[u]&{\bullet}^{y_3}\ar@{-}[ul]& {\bullet}^{y_4}\ar@{-}[ull]& \\
		&&{\bullet}^{z_3}\ar@{-}[u]\ar@{-}[d]&&& \\ &&\vdots &&& }
		\end{equation*}
		Now, $Min(Y) = \{ y_1,y_2,y_3,y_4\}$ and by Theorem \ref{Main}, 
\begin{equation*}
		Soc_{l}({I(X,R)})=\bigoplus_{x\in Min(Y)}S(x,Soc_{l}(R)) = \bigoplus_{i=1}^4 S(y_i,Soc_{l}(R)).
\end{equation*}
\end{example}

Now we will use the minimal left ideals of the ring to calculate the socle.  
\begin{lemma} \label{minimal}
Let $R$ be an arbitrary ring and $A$ be a minimal left ideal. Let $X$ be a locally finite partially ordered set with $x \in Min(X)$. Then for any $g \in S(x,A)$, the left ideal generated by $g$ is a minimal left ideal of $I(X,R)$. Moreover
\begin{equation*}
\bigoplus_{g \in S(x,A)} \langle g \rangle _l = S(x,A).
\end{equation*}
\end{lemma}
\begin{proof}
Consider $g \in S(x,A)$, where $x \in Min(X)$, 

for any $u,v \in X$, $g(u,v) = 
\begin{cases}
a_{v} \mbox{ if } u=x, v \geq u \\ 
0 \mbox{ if } u \neq x
\end{cases}
$ where $a_v \in A$. We would like to show that the left ideal generated by $g$ 
\\

\{ag : 
$\begin{array}{cl}
g(u,v) \in A & \mbox{ if } u=x, v \geq u \\ 
g(u,v) = 0 & \mbox{ if } u \neq x,%
\end{array}
a \in R \}$	
is a minimal
left ideal of ${I(X,R)}$. 
One can write the left ideal generated by $g$ as:
\begin{equation*}
\langle g \rangle _l=\{f \in I(X,R):f(x,v)=aa_v \mbox { for some }  a \in R \}.
\end{equation*}
Fix some $f \in \langle g \rangle _l$ such that $f(x,v)=aa_v$ for some $a \in R$. Then
$\langle f \rangle _l=\{h \in I(X,R):\begin{array}{cl}
h(u,v) = 0 & \mbox { if } u \neq x, \\ 
h(x,v) = raa_v  & \mbox { if } v \geq x  %
\end{array}
\mbox { for some }  a \in R \}.$ Since $a_v \in A$ and $A$ is a minimal left ideal of $R$, $aa_v \in A$. For any $r \in R$, we get $raa_v \in \langle aa_v \rangle _l \subseteq A$, hence $\langle aa_v \rangle _l=A$. For each $a_v \in A$, we can generate $A$ with $\langle aa_v \rangle _l$. That is, there exists $r_0 \in R$ such that $a_v=r_0aa_v$. For any $g \in S(x,A)$, $g(x,v)=a_v=r_0aa_v=r_0f(x,v)$. Hence
\begin{equation*}
\begin{array}{cl}
g(x,v)=r_0f(x,v) & \mbox{ for all } v, \\ 
g(u,v) = 0 =r_0f(u,v) & \mbox{ if } u \neq x.%
\end{array}
\end{equation*}
So $g = r_0e_{xx}f \in \langle f \rangle _l$. Thus 
\begin{equation*}
\langle g \rangle _l \subseteq \langle f \rangle _l \subseteq \langle g \rangle _l.
\end{equation*}
That is $\langle f \rangle _l= \langle g \rangle _l $. Hence we proved the left ideal generated by $g$ is a minimal
left ideal of ${I(X,R)}$. 
\end{proof}	
\begin{proposition} \label{main}
	Assume $R$ is a ring with finitely many minimal left ideals and $X=\mathbb{Z}^+$. Then $Soc_l(I(\mathbb{Z}^+,R))=S(1,Soc_l(R))$.
\end{proposition}
\begin{proof}
By Lemma \ref{C(I)}(iii), $Soc_l(I(\mathbb{Z}^+,R)) \subseteq S(1,Soc_l(R))$. For the converse inclusion, we use Lemma \ref{minimal}. 
Since $Soc_l(I(X,R))$ is the sum of all minimal left ideals of $I(X,R)$,
\begin{equation*}
\bigoplus_{A} S(1,A) \subseteq Soc_l(I(X,R))
\end{equation*} where $A$ ranges over all minimal left ideals of $R$. 

Now, we will prove that 
if $R$ has finitely many minimal left ideals $\{A_i \}_{i=1}^n$, then
\begin{equation*}
S(1,Soc_l(R)) = \bigoplus_{i=1}^n S(1,A_i)  
\end{equation*} and hence
\begin{equation*}
S(1,Soc_l(R)) = Soc_l(I(X,R)).
\end{equation*} 

Since $R$ has finitely many minimal left ideals $\{A_i \}_{i=1}^n$, $ Soc_l(R) = \bigoplus_{i=1}^n A_i$.
Given any $f \in S(1,Soc_l(R))$ and $y \in X$, we have $f(1,y) \in Soc_l(R) = \bigoplus_{i=1}^n A_i$. That is 
\begin{equation*}
f(1,y)= \sum_{i=1}^{n} {a^y_{i}}
\end{equation*} where $a^y_{i} \in A_i$. Define the functions  $f_i \in I(\mathbb{Z}^+,R)$ for $i=1, ..., n$ as: 
\begin{equation*} f_i(x,y)= \begin{cases}
 0  & \text{ if } x \neq 1 \\ 
 a^y_{i} & \text{ if }  x = 1, 
\end{cases} 
\text{ for any } y \in \mathbb{Z}^+. 
\end{equation*} 
So $f_i \in S(1,A_i)$, hence $ f= \sum_{i=1}^{n}{f_i} \in \bigoplus_{i=1}^n S(1,A_i)$. Clearly, $\bigoplus_{i=1}^n S(1,A_i)$ is contained in $S(1,Soc_l(R))$. This completes the proof.
\end{proof}

We want to state that when $R$ is chosen to be a field or $\mathbb{Z}$, socle and singular ideal are easy to compute. 

\begin{example}
Let $X=\mathbb{Z}^+$, then $Min(X) = \{ 1\}$ is a singleton and $Min(X)$ is
a maximal antichain, but $\kappa(x)$ is not finite for any $x$.

\begin{itemize}
\item if $R =\mathbb{Z} $, $Sing(\mathbb{Z}) = \{ 0\} = Soc(\mathbb{Z})$. 
\newline
By Proposition~\ref{Nonsingular}, $Sing((I,\mathbb{Z})) = \{ 0\}$ and by
Lemma~\ref{C(I)}, 

$Soc_l(I(X,\mathbb{Z})) \subset S(1,Soc(\mathbb{Z})) = \{0\}$.

\item $R=F$ is a field, then $Sing(F) = \{ 0\}$, $Soc(F)=F$. \newline
By Proposition~\ref{Nonsingular}, $Sing((I,F)) = \{ 0\}$ and since $F$ has a unique minimal (left) ideal $F$ itself by Proposition~\ref{main}, 
\begin{equation*}
Soc_l(I(X,F)) = S(1,F).
\end{equation*}
\end{itemize}
\end{example}


\bigskip

We conclude this section by describing the right socle of ${I(X,R)}$ and give an example of an incidence algebra having the same left and right socle. 

For any $x \in Max(X)$, fix a right ideal $I_x$ of $R$ and define $%
T(x,I_x)$ as 
\begin{equation*}
T(x,I_x) = \{f \in I(X,I_x) : \, f(u,v) = 0 \mbox{ if } v \neq x \}.
\end{equation*}
Then, we form 
\begin{equation*}
B({I(X,R)}) := \bigoplus_{x \in Max(X)}T(x,I_x) = \bigoplus_{x \in Max(X)}
\{f \in I(X,I_x) : \, f(u,v) = 0 \mbox{ if } v \neq x \}
\end{equation*}
which is a two-sided ideal of ${I(X,R)}$.

Equivalent to Lemma \ref{C(I)}, we have the following results.

\begin{lemma}
\label{B(I)} Assume $X$ is a locally finite partially ordered set where $%
Max(X)$ is a maximal antichain and $\{I_x : x \in Max(X)\}$ is a collection
of right essential ideals of $R$.

\begin{enumerate}
\item $B({I(X,R)})$ is a right essential ideal of ${I(X,R)}$.

\item Let $B$ be the intersection of all right essential ideals of the form $%
B({I(X,R)})$. Then 
\begin{equation*}
B = \bigoplus_{x \in Max(X)} T(x,Soc_r(R)).
\end{equation*}

\item $Soc_r({I(X,R)}) \subset B$.
\end{enumerate}
\end{lemma}


\medskip 

Under the assumption that $Max(X)$ is a maximal antichain and $\lambda(x)$
is finite for any $x \in Max(X)$, the description of the right socle follows.

\begin{theorem}
\label{MaxSocle} Assume $X$ is a locally finite partially ordered set where $%
Max(X)$ is a maximal antichain and $\lambda(x)$ is finite for any $x \in
Max(X)$. Then, 
\begin{equation*}
Soc_r({I(X,R)}) = \bigoplus_{x \in Max(X)} T(x,Soc_r(R))
\end{equation*}
\end{theorem}


\medskip 

When $X$ is a finite partially ordered set, the hypothesis of Theorem
\ref{MinSocle} and Theorem \ref{MaxSocle} are satisfied. As an immediate
corollary, the left and the right socle of the incidence algebra are given explicitly
as 
\begin{equation*}
Soc_l({I(X,R)}) = \bigoplus_{x \in Min(X)} S(x,Soc_l(R)),
\end{equation*} 
\begin{equation*}
Soc_r({I(X,R)}) = \bigoplus_{x \in Max(X)} T(x,Soc_r(R)).
\end{equation*}
This leads to constructing examples of incidence algebras with distinct left and right socles:


\begin{example}
\label{Exfinite} Let $X =\{x,y,z\}$ be a partially ordered set with the
non-reflexive relations $x \leq y $, $x \leq z$. Hence, Hasse diagram of $X$
is 
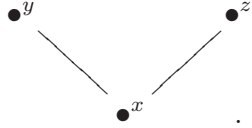
\begin{equation*}
\xymatrix{{\bullet}^{y}&&{\bullet}^{z} \\
&{\bullet}^{x}\ar@{-}[ul]\ar@{-}[ur]&.}
\end{equation*}
By mapping each $f \in I(X,R)$ to $\left [ 
\begin{array}{ccc}
f(x,x) & f(x,y) & f(x,z) \\ 
0 & f(y,y) & 0 \\ 
0 & 0 & f(z,z)%
\end{array}
\right ]$,

${I(X,R)}$ is isomorphic to $R$-algebra 
\begin{equation*}
\left( 
\begin{array}{ccc}
R & R & R \\ 
0 & R & 0 \\ 
0 & 0 & R%
\end{array}
\right).
\end{equation*}
By Theorem~\ref{MinSocle}, $Soc_l({I(X,R)}) = \left( 
\begin{array}{ccc}
Soc_l(R) & Soc_l(R) & Soc_l(R) \\ 
0 & 0 & 0 \\ 
0 & 0 & 0%
\end{array}
\right).$ 
By Theorem~\ref{MaxSocle}, $Soc_r({I(X,R)}) = \left( 
\begin{array}{ccc}
0 & Soc_r(R) & Soc_r(R) \\ 
0 & Soc_r(R) & 0 \\ 
0 & 0 & Soc_r(R)%
\end{array}
\right).$ 

Hence, $Soc_l({I(X,R)}) \neq Soc_r({I(X,R)})$.

It is well-known that the left singular ideal of any ring is contained in
the right annihilator of its left socle. Moreover, when the ring is left
Artinian, then the left singular ideal of the ring is exactly the right
annihilator of the left socle (eg. see \cite[Proposition 2.1.4]{Lam:Lectures}). In \cite[Ch.8, p.305]{SO:IncAlg}, it is stated that ${I(X,R)}$ is left Artinian if and only if $X$ is finite and $R$ is left
Artinian. In view of all these results, let $R$ be an Artinian ring, thus we further provide an example such that 
$Sing_l({I(X,R)}) = ann_r(Soc_l({I(X,R)}))$.

By Corollary~\ref{CorMinSingular}, 
\begin{equation*}
Sing_l({I(X,R)}) = I(X,Sing_l(R))= \left( 
\begin{array}{ccc}
Sing_l(R) & Sing_l(R) & Sing_l(R) \\ 
0 & Sing_l(R) & 0 \\ 
0 & 0 & Sing_l(R)%
\end{array}
\right) .
\end{equation*}
Note that $Soc_l({I(X,R)}) Sing_l({I(X,R)}) = 0$.
\end{example}




\medskip

The next proposition gives the necessary conditions for the
left and the right socle of an incidence algebra to be the same.

\begin{proposition}
\label{Socl=Socr} 
If $Soc_{l}(R)=Soc_{r}(R)$ and $X$ is an antichain then $Soc_{l}({I(X,R)}%
)=Soc_{r}({I(X,R)}).$ Moreover, in this case, the left/right socle of ${%
I(X,R)}$ is 
\begin{equation*}
Soc({I(X,R)})=\displaystyle\bigoplus_{x\in X}Soc(R)
\end{equation*}%
where $Soc(R):=Soc_{l}(R)=Soc_{r}(R)$.
\end{proposition}

\begin{proof}
Assume $Soc_l(R)=Soc_r(R)$ and $X$ is an antichain. Then $\kappa(x) = 1=
\lambda(x)$ and Lemma~\ref{propPoset}(iii), $Min(X) = X = Max(X)$ is a
maximal antichain. By Proposition~\ref{MinSocle} and Proposition~\ref{MaxSocle}, 
\begin{equation*}
Soc_l({I(X,R)}) =\bigoplus_{x \in Min(X)} \{f \in I(X,Soc_l(R)) : \, f(u,v)
= 0 \mbox{ if } u \neq x\} 
\end{equation*}
\begin{equation*}
\cong \bigoplus_{x \in X} Soc_l(R)
= \bigoplus_{x \in X} Soc_r(R) \end{equation*}
\begin{equation*}\cong \bigoplus_{x \in Max(X)} \{f \in
I(X,Soc_r(R)) : \, f(u,v) = 0 \mbox{ if } v \neq x\} = Soc_r({I(X,R)})
\end{equation*}
Hence, in this case, the socle of the incidence algebra is the direct sum of 
$|X|$-many copies of the socle of $R$. ie. $Soc({I(X,R)}) = \bigoplus_{x \in
X} Soc(R) $.
\end{proof}


\begin{remark}
\bigskip Note that the left and the right socle of an incidence algebra might be equal when $X$ is not an antichain. 
Consider the case $X=\mathbb{Z}.$ 
Since $Min(X)=\emptyset $ by Lemma \ref{Zn}, $Soc_{l}(I(X,R))=\left\{
0\right\} $ and analogously $Max(X)=\emptyset ,$ $Soc_{r}(I(X,R))=\left\{
0\right\}.$
\end{remark}

{\bf Acknowledgement} \\
The authors are deeply thankful to Nesin Math Village, \c{S}irince, Izmir for providing an excellent research environment where this work has been completed. This research question has been motivated by an idea of Prof.Eugene Spiegel of University of Connecticut while the first author was conducting her Ph.D. under his supervision. We dedicate this long time overdue paper to his memory. 


\end{document}